\documentclass[11pt, reqno]{amsart}

\usepackage[english]{babel}
\usepackage[latin1]{inputenc}
\usepackage{babel}
\usepackage{fancyhdr}
\usepackage{amsmath,amsfonts,amssymb,amsthm}
\usepackage[mathcal]{euscript}
\usepackage{mathrsfs}
\usepackage[tmargin=1.5in,bmargin=1.5in,lmargin=1.3in,rmargin=1.3in]{geometry}
\usepackage[all]{xy}
\usepackage{textcomp}
\usepackage{epsfig}
\usepackage{graphics}
\usepackage{graphicx}
\usepackage{mathrsfs}

\usepackage{caption}
\usepackage[OT2,OT1]{fontenc}
\newtheorem{thm}{Theorem}[section]

\newtheorem{cond}[thm]{Condition}

\newtheorem{lem}[thm]{Lemma}
\newtheorem{prop}[thm]{Proposition}

\newtheorem{defi}[thm]{Definition}

\usepackage{microtype}
\usepackage{bbm}
\usepackage{color}
\usepackage[textsize=small]{todonotes}       

\definecolor{halfgray}{gray}{0.55}
\definecolor{webgreen}{rgb}{0,.5,0}
\definecolor{webbrown}{rgb}{.6,0,0}
\definecolor{Maroon}{cmyk}{0, 0.87, 0.68, 0.32}
\definecolor{royalblue}{cmyk}{1, 0.50, 0, 0}
\definecolor{Black}{cmyk}{0, 0, 0, 0}

\usepackage{hyperref}
\hypersetup{%
    colorlinks=true, linktocpage=true, pdfstartpage=1, pdfstartview=FitV,%
    breaklinks=true, pdfpagemode=UseNone, pageanchor=true, pdfpagemode=UseOutlines,%
    plainpages=false, bookmarksnumbered, bookmarksopen=true, bookmarksopenlevel=1,%
    hypertexnames=true, pdfhighlight=/O,
    urlcolor=webbrown, linkcolor=royalblue, citecolor=webgreen, 
}

\numberwithin{equation}{section}

\newcommand{\sss}{\scriptscriptstyle}

\renewcommand{\P}{\mathbb{P}}

\newcommand{\indi}{I}
\newcommand{\Ccal}{\mathcal{C}}
\newcommand{\Acal}{\mathcal{A}}
\newcommand{\Ncal}{\mathcal{N}}

\newcommand{\Lcal}{\mathscr{L}}

\newcommand{\dg}{\boldsymbol d}

\newcommand{\Unex}{\mathcal{U}}

\newcommand{\Dcal}{\mathcal{D}}

\newcommand{\elltwo}{\ell_{\neq 2}}

\newcommand{\Cmax}{\mathscr{C}_{\rm max}}
\newcommand{\cluster}{\mathscr{C}}
\newcommand{\Poi}{{\sf Poi}}
\newcommand{\E}{\mathbb{E}}
\newcommand{\e}{\mathrm{e}}

\newcommand{\CMd}{\mathrm{CM}_n(\boldsymbol{d})}
\newcommand{\CMdd}{\mathrm{CM}_{n-n_2}(\boldsymbol{d'})}
\newcommand{\CMdt}{\mathrm{CM}_{n-m_i}(\tilde{\dg})}
\newcommand{\Cmaxc}{\Cmax^{\sss \mathtt Cycle}}
\newcommand{\Cmaxl}{\Cmax^{\sss \mathtt Line}}
\newcommand{\Cjl}{\cluster_j^{\sss \mathtt Line}}
\renewcommand{\P}{\mathbb{P}}
\newcommand{\prob}{\mathbb{P}}
\newcommand{\op}{o_{\sss \prob}}
\newcommand{\Op}{O_{\sss \prob}}
\newcommand{\Tp}{\Theta_{\sss \prob}}

\newcommand{\dto}{\overset{d}{\rightarrow}}
\newcommand{\pto}{\overset{\sss\prob}{\rightarrow}}

\newcommand{\eqn}[1]{\begin{equation}#1\end{equation}}
\newcommand{\eqan}[1]{\begin{align}#1\end{align}}
\newcommand{\nn}{\nonumber}

\newcommand{\ddsum}{\sideset{_{}^{}}{_{}^{*}}\sum}

\begin{document}

\title{Almost-2-regular random graphs}
\author{Lorenzo Federico}
\maketitle

\begin{abstract}

We study a special case of the configuration model, in which almost all the vertices of the graph have degree $2$. We show that the graph has a very peculiar and interesting behaviour, in particular when the graph is made up by a vast majority of vertices of degree $2$ and a vanishing proportion of vertices of higher degree, the giant component contains $n(1-o(1))$ vertices, but the second component can still grow polynomially in $n$. On the other hand, when almost all the vertices have degree $2$ except for $o(n)$ which have degree $1$, there is no component of linear size.
\end{abstract}
\section{Introduction}
In this paper, we study the behaviour of random graphs in which the majority of vertices have degree $2$, but there exists a vanishing proportion of vertices of either higher or lower degree. The study of random graphs with a prescribed degree sequence has been a very popular topic for decades. The usual way to construct them has been the procedure known as the \emph{configuration model}, which yields a random multigraph over $n$ vertices with a fixed degree sequence $\dg=\{d_1,d_2,...,d_n\}$. Conditioning on simplicity, the configuration model has uniform distribution over all the labeled graphs with such degree sequence. In the configuration model each vertex $v_i$, $i \in [n]$ is given $d_i$ half-edges, and all the half-edges are paired uniformly at random to create the edges of the graph. This model is very flexible, allowing for any arbitrary degree sequence, while remaining tractable, thanks to its high level of symmetry. For these reasons it has become widely popular, both among applied researchers trying to fit the model to empirical degree distributions observed in real-world networks, such as power laws, and theoretical mathematicians who were looking for a tractable model that could exhibit a wide variety of behaviours.
The study of connectivity of random graphs with a prescribed degree sequence started with the work of Bollob\'as \cite{Boll01} and Wormald \cite{Wor81}, who proved that if the minimum degree is at least $3$, the graph is connected with high probability (w.h.p.), and then it was further developed by \L uczak \cite{Luc92}, and the author and van der Hofstad \cite{FedHof16}, who analyzed the asymptotic connectivity probability when vertices of lower degree are allowed.
Moreover, it is known that the configuration model, like many other random graphs, presents a phase transition. The first investigation of this phenomenon is due to Molloy and Reed \cite{MolRee95}, who proved that the critical point for the phase transition is identified by $\sum_{i=1}^nd_i(d_i-2)/n= 0$. Above such threshold there exists w.h.p. a connected component that contains a positive fraction of the vertices, below, instead, the size of the largest component is determined by the highest degree vertex up to a constant, as proved by Janson \cite{Jan08}. For the study of the behaviour close to criticality see the work by Dhara et al. \cite{DhaHofLeeSen16a, DhaHofLeeSen16b} and van der Hofstad et al \cite{HofJanLuc18}.

In all these papers the situation in which vertices of degree $2$ make up almost the entirety of the graph has been left out as a boundary case that was hard to address with the usual techniques. Here, we will analyze in detail what happens under such an assumption, showing that indeed these graphs show very peculiar properties that are rarely found in other settings. 

	\subsection*{ Notation.}
All limits in this paper are taken as $n\to \infty$ unless stated otherwise.
For asymptotic statements we use the following notation:
\begin{itemize}
\item Given a sequence of events $(\mathcal{A}_n)_{n \geq 1}$ we say that $\mathcal A_n$ happens \emph{with high probability (w.h.p.)} if $\P(\mathcal{A}_n) \to 1$.
\item Given the random variables $(X_n)_{n \geq 1}, X$, we write $X_n \dto X$ and $X_n \pto X$ to denote convergence in distribution and in probability, respectively.
\item For sequences of (possibly degenerate) random variables  $(X_n)_{n \geq 1}$, $(Y_n)_{n \geq 1}$, we write $X_n=O(Y_n)$ if the sequence $(X_n/Y_n)_{n \geq 1}$ is bounded almost surely; $X_n=o(Y_n)$ if $X_n/Y_n \to 0$ almost surely; $X_n =\Theta(Y_n)$ if $X_n=O(Y_n)$ and $Y_n=O(X_n)$ .
\item Similarly, for sequences $(X_n)_{n \geq 1}$, $(Y_n)_{n \geq 1}$ of (possibly degenerate) random variables, we write $X_n=\Op(Y_n)$ if the sequence $(X_n/Y_n)_{n \geq 1}$ is tight; $X_n=\op(Y_n)$ if $X_n/ Y_n \pto 0$; and $X_n =\Tp(Y_n)$ if $X_n=\Op(Y_n)$ and $Y_n=\Op(X_n)$.  
\item $\Poi(\lambda)$ denotes a Poisson distributed random variable with mean $\lambda$ and ${\sf Bin} (n,p)$ denotes a random variable with binomial distribution with $n$ trials each with probability of success $p$.
\end{itemize}

We will use the standard abbreviations i.i.d. for independent identically distributed and a.s. for almost surely.

\section{Main Results}

We first recall that the configuration model $\CMd$ is constructed from a degree sequence $\dg=\{d_1,d_2,...,d_n\}$ such that $\sum_{i=1}^nd_i=:\ell$ is even, giving to each vertex $v_i$ for any $i \in [n]$, $d_i$ half-edges and pairing them uniformly at random to form the edges. We consider the pairing as a sequential process, i.e., we pick one of the $\ell$ half-edges, and we pair it with another one chosen uniformly at random among the $\ell-1$ remaining to form an edge, then we pick another half edge and pair it with a uniformly chosen one among the $\ell-3$ remaining ones to build another edge, and we iterate until all the half-edges have been paired. The specific rule according to which the new half-edge is chosen at the beginning of each new step is not relevant for the distribution of the final outcome, something that we will exploit in many proofs.
We study separately two cases depending on whether the degree sequence allows for vertices of degree larger than $2$ and for vertices of degree $1$. We always assume that there are no vertices of degree $0$ since they do not interact with the rest of the graph and are thus uninteresting.
For each $j \in \mathbb N$ we define $\Ncal_j$ as the set of vertices of degree $j$ in $\dg$ and $n_j$ as its cardinality. 
We write the total number of half-edges as
\eqn{
\ell:=\ell(n)=\sum_{i=1}^n d_i= \sum_{j=1}^\infty jn_j,
}
and $\Lcal$ for the set of all the half-edges. Moreover, we define the total number of half-edges attached to vertices of degree different from $2$ as

\begin{equation}
\elltwo :=\sum_{j\neq 2} j n_j,
\end{equation}
and $\Lcal_{\neq 2}$ for the set of all such half-edges.
Note that if $\ell$ is even, then also $\elltwo$ is even, since $\elltwo=\ell-2n_2$.
We study separately the two cases in which the degree sequence allows for vertices of degree higher or lower than $2$, as their behaviour is  completely different.
 We define the conditions under which we call a sequence of configuration models $\CMd$ a sequence of \emph{upper almost-two-regular graphs}.

\begin{cond}[Upper almost-two-regular graph]
\label{2-up}
We define a sequence of random graphs as a sequence of upper almost-two-regular graphs if it is distributed as a sequence $(\CMd)_{n\in \mathbb N}$  and the following conditions are satisfied as $n \to \infty$:

\begin{itemize}
\item $\ell/n\to 2,$
\item $n_0,n_1=0,$
\item $\elltwo \to \infty$.
\end{itemize}
\end{cond}
This conditions can be described as the degree sequence having a vast majority of vertices of degree $2$ and a diverging but sublinear number of vertices of higher degree.

It was mentioned by van der Hofstad \cite{FedHof16, Hofs17} that in the special case in which only vertices of degree $2$ and $4$ were allowed, the largest component contains almost all the vertices in the graph. We generalize and strengthen this result in the main theorem of this paper, which describes the asymptotyc behaviour of the two largest components, $\Cmax$ and $\cluster_2$:

\begin{thm}Consider a sequence $\CMd$ that satisfies Condition \ref{2-up}. Then as $n \to \infty$, 
\label{thm:MainUp}
\eqn{
n-|\Cmax|= \Op(n/\elltwo),\label{main:first}
}
and
\eqn{
\dfrac{|\cluster_2|\ell_{\neq 2}}{n}\dto Y_2,\label{main:second}
}
where $Y_2$ is a random variable with support on $\mathbb{R}^+$ and cumulative distribution function

\begin{align}
F_{Y_2}(a)=\exp\Big\{-\int_a^\infty \frac{\e^{-2r}}{2r}dr\Big\}.
\end{align}

\end{thm}

We immediately note how this theorem implies the possibility for a supercritical random graph to have components outside the giant of every intermediate order of magnitude between $1$ and $ n$, choosing an appropriate asymptotic of $\elltwo$. This is thus a counterexample to the universality of the ``no middle ground" property,  i.e. that in a supercritical random graph the connected components outside the giant are always very small, typically with the size of the second largest component expressed as a polynomial in $\log n$, which holds for the most famous random graph models, as we will discuss in the next section.

We also define the \emph{lower almost-two-regular graphs} as sequences of configuration models satisfying the following conditions as $n \to \infty$:

\begin{cond}[Lower almost-two-regular graph]
\label{2-lo}
We define a sequence of random graphs as a sequence of lower almost-two-regular graphs if it is distributed as a sequence $(\CMd)_{n\in \mathbb N}$  and the following conditions are satisfied as $n \to \infty$:

\begin{itemize}
\item $\ell/n\to 2,$
\item $n_0=0,$
\item $\elltwo=n_1\to \infty$.
\end{itemize}
\end{cond}
In this case we are instead considering degree sequences in which the vast majority of vertices have degree two, and the only other ones are a sublinear number of degree $1$ vertices. In this case the behavior of the graph is radically different, as expressed in the following theorem. We define $\cluster_j$ as the $j$-th largest cluster, then we obtain:

\begin{thm}Consider a sequence $\CMd$ that satisfies Condition \ref{2-lo}. Then as $n \to \infty$, for  every $j\in \mathbb N$
\label{thm:MainLow}
\begin{equation}
|\cluster_j|=\dfrac{2n\log n_1}{n_1}(1+\op(1)).
\end{equation}

\end{thm}

We see how this  model closely resembles subcritical random graphs, with the largest components being all very similar to each other and with no component of linear size.

We do not investigate in this paper what happens when the vast majority of vertices have degree $2$ but there are both vertices of higher and lower degree. That would require a very detailed case analysis depending on the relative scaling of $n_i$ for different $i$s as $n \to \infty$ which we think could be interesting for a future work.

\section{Discussion of the results}

In this section, we present some relevant consequences and some interesting observations related to our main theorems. 
The first interesting result to point out is the fact that the upper almost-two-regular graph is a counterexample to the very general ``no middle ground" property of supercritical random graphs.
It was first observed by  Erd\H{o}s and R\'enyi \cite{ErdRen60} that the random graph $G(n,m)$ (the uniform random graph with $n$ vertices and $m$ edges), in the supercritical phase, i.e., when $m>(1+c)n/2$, presents a big gap between the largest component, which contains a positive fraction of the vertices, and the second one, which contains $\Tp(\log n)$ vertices. Such property has been established to hold for a huge class of random graphs such as inhomogeneous random graphs \cite{BolJanRio07}, most of the other cases of the configuration model \cite{MolRee95} and percolation on the Hamming graph \cite{HofLucSpe10}. For percolation on a large box or torus on $\mathbb Z^d$, the second largest component in the supercritical phase is of order $\Tp(\log n^{d/(d-1)})$ \cite{KesZha90}. In the case presented in this paper instead, we see from Theorem \ref{thm:MainUp} that the upper almost-two-regular graph shows a clear supercritical behaviour, that is, there exist a unique giant component, that has a clearly different structure from all the other components and includes a positive fraction (actually, almost all) of the vertices, but the second largest component can be of every possible order of magnitude such that $1\ll |\cluster_2|\ll n$, choosing the right scaling between $n$ and $\elltwo$.
\begin{figure}[h]
\includegraphics[width=0.4\textwidth]{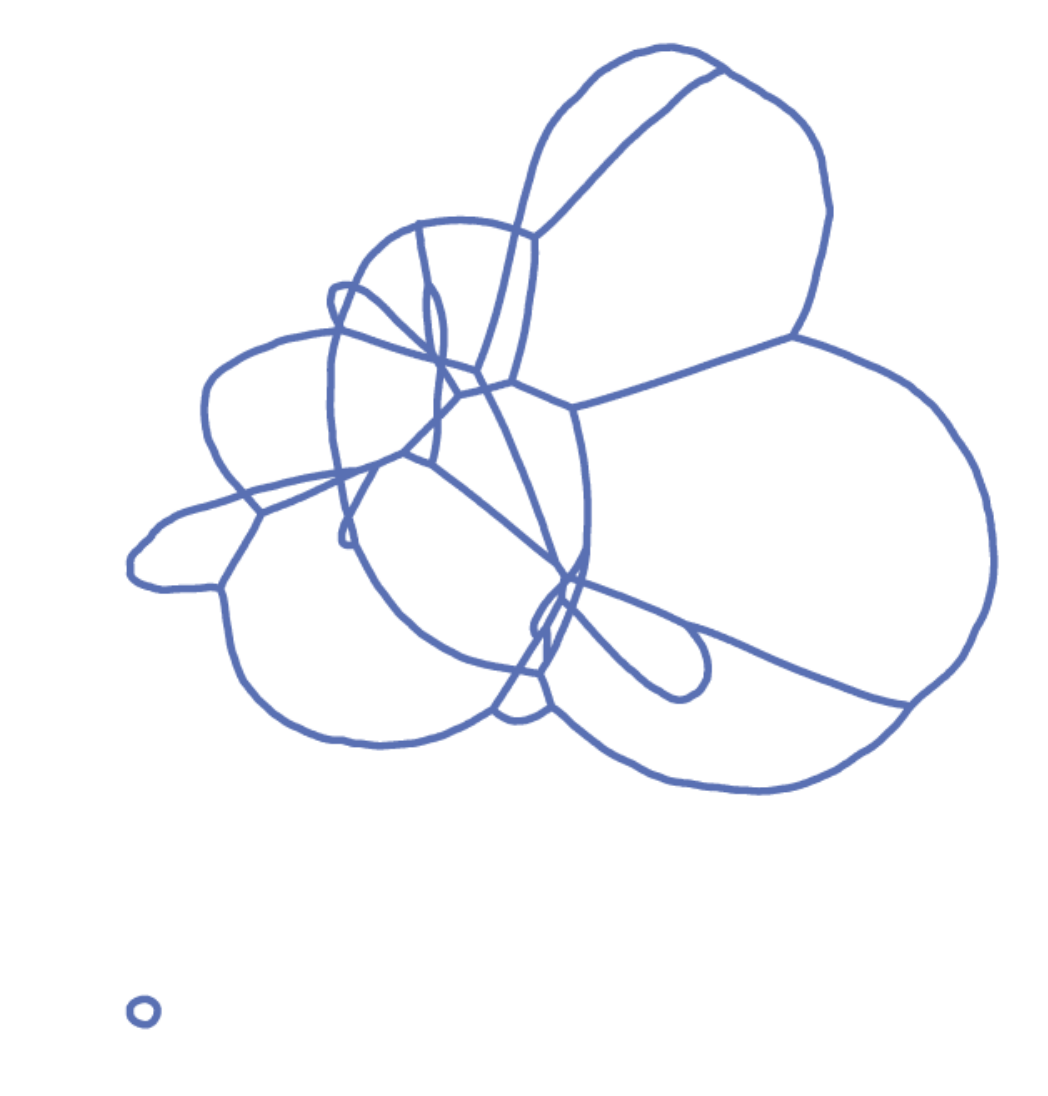}
\includegraphics[width=0.4\textwidth]{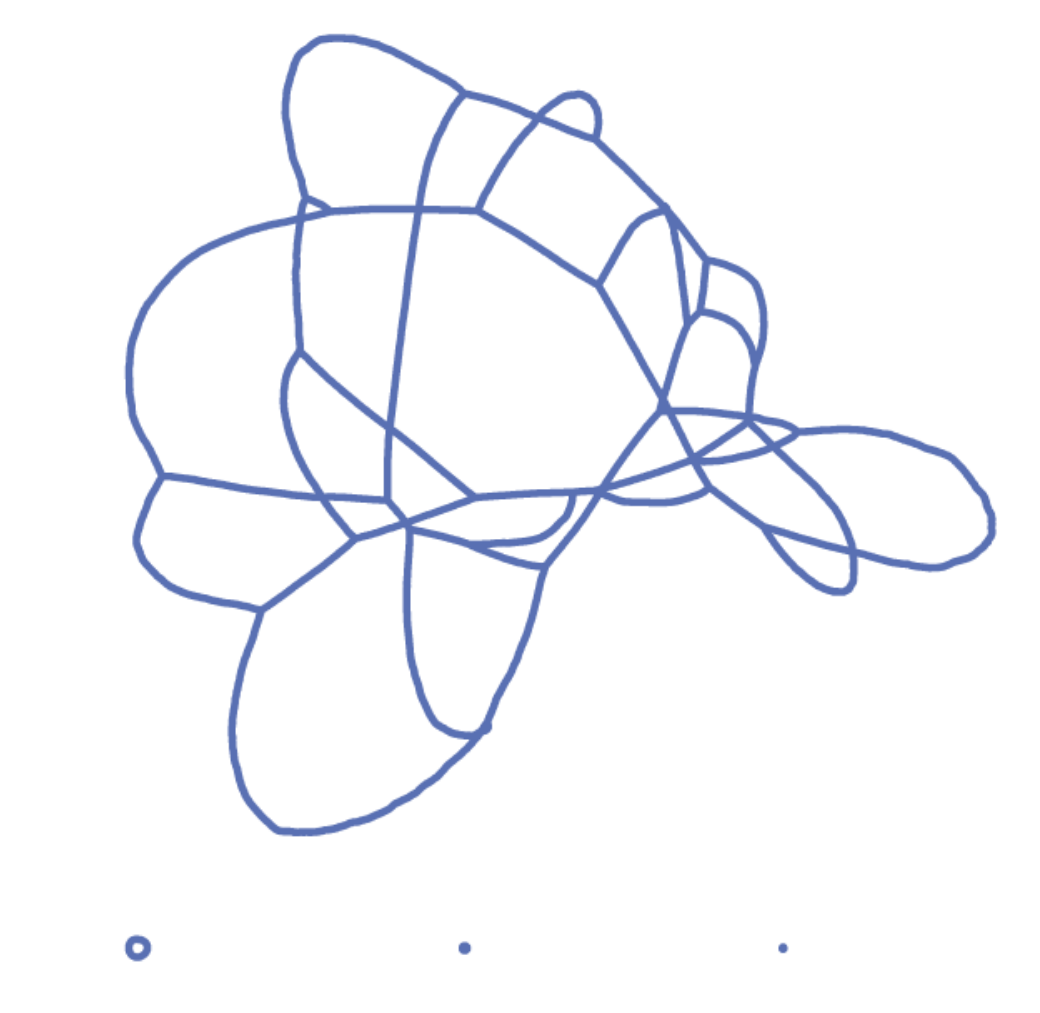}
\caption{2 samples of the Configuration model over 10000 vertices with $n_2=9970$, $n_3=30$. It is easy to observe that the graph is mostly made by long linear sections.}
\label{fig:onedime}\end{figure}
One of the possible reasons for this phenomenon is the \emph{one-dimensional nature} of the graph. We can observe that the \emph{local weak limit} of the graph is $\mathbb Z$, i.e., the neighbourhood of a uniformly chosen vertex resembles w.h.p. a straight line up to any finite distance (see Figure \ref{fig:onedime}). The reasons for which a large second component in a supercritical graph is unlikely is that large sets tend to have large boundaries, and thus are hard to separate from the giant component. In a one-dimensional model instead, connected sets with large boundaries are impossible to achieve, and thus it is relatively easy to separate even quite large components from the giant.

The lower almost-two-regular graph instead, shows clear signs of being a subcritical graph, i.e., the largest component is in no way ``special", but the $k$ largest components are almost indistinguishable for every fixed $k$. Still, also in this case we see the peculiar nature of this model, as the size of the largest component is in no way determined by the highest degree up to a constant or at most a logarithmic term, as it is in most cases, like $G(n,m)$ \cite{ErdRen59}, the Random Intersection Graph \cite{Beh07} or many instances of the Configuration Model \cite{Jan08}  but we can observe polynomially large components in a subcritical graph with highest degree equal to $2$.

\section{Proof of Theorem \ref{thm:MainUp}}

In this section, we prove the main results about the upper almost-2-regular graph, summarized in Theorem \ref{thm:MainUp}. Our proof will consist of two main steps: first, we prove that w.h.p. all the connected components except the giant are cycles (i.e. components that contain only vertices of degree 2), second, we find the distribution of the size of the largest cyclic component in the graph. The first proof uses an argument similar to the one presented by \L uczak in \cite{Luc92}, by which the configuration model after the removal of the vertices of degree $2$ is equivalent to a configuration model with a different degree sequence, the second instead will require the use of enumerative combinatorics to find the probability of having cycles of a given length. In order to understand the large-scale structure of $\CMd$, we define the notion of \emph{kernel} of the graph, adapted from \cite{Luc92}.

\begin{defi}[Kernel of the Configuration Model]
We define the kernel $K(\CMd)$ of the configuration model as the graph obtained from $\CMd$ after running the following algorithm:

\begin{itemize}
\item[\tt Initialize] Sample $\CMd$.
\item[\tt Step] At every step, choose a vertex $v_i$ such that $d_i=2$ according to any arbitrary rule and remove it from $\CMd$. 
\textbf{If} $v_i$ has a self loop, do nothing more, \textbf{else} pair the two half-edges which had been paired with the half-edges of $v_i$, even if this procedure would create a self loop or a multi-edge.
\item[\tt Terminate] Stop the process when there are no more vertices of degree $2$.
\end{itemize}
\end{defi}
From the definition, it follows that $K(\CMd)$ is a multigraph with a degree sequence $\dg'$ such that, defining $n'_i$ as the number of vertices of degree $i$ in $\dg'$, we have $n'_2=0$, $n'_i=n_i$ for all $i \neq 2$. Moreover, in the next lemma we prove that $K(\CMd)$ is distributed as $\CMdd$.

\begin{lem}Consider $\CMd$ with any arbitrary degree sequence $\dg$ and its kernel $K(\CMd)$. Then
\label{lem:KernelDistribution}
\eqn{
K(\CMd)\overset{d}{=}\CMdd,
}
where $\dg'$ is such that $n'_2=0$, $n'_i=n_i$ for all $i \neq 2$.
\end{lem}
\proof

We show that however we choose an half-edge $e_1$ in $\Lcal_{\neq 2}$, $e_1$ is paired in $K(\CMd)$ with another half-edge $e_2$ chosen uniformly in $\Lcal_{\neq 2}$. We do so through an algorithm that sequentially builds $\CMd$. 
Define the time variable of the algorithm as a pair $(i,j)$, $i\geq 1, j\geq -1$, and the sets of unexplored half-edges at any time, $\Unex(i,j)$. We describe in the following pseudocode how the algorithm can sample the edge set of $K(\CMd)$ as a set of pairs of half edges $(\{e_1(i),e_2(i)\})_{i=1}^{\elltwo/2}$, while  building $\CMd$ at the same time. 

\begin{enumerate}
\item[\tt Initialize] Set $i=1$, $j=-1$. Set $\Unex (1,-1)=\Lcal$. 
\item[\tt Step] \begin{enumerate}
\item \textbf{If} $\Lcal_{\neq 2}\cap \Unex(i,-1)=\varnothing$, pair the remaining half-edges in $\Unex(i,-1)$ uniformly at random and terminate the algorithm.  

\textbf{Else}, choose $e \in \Lcal_{\neq 2}\cap \Unex(i,-1)$ according to any arbitrary rule. Set $j=0$, $\Unex(i,0)=\Unex(i,-1) \setminus \{e_1\}$ and $e_1(i)=e$
\item Choose an half-edge $e'$ uniformly at random in $\Unex(i,j)$ and pair it with $e$ 
\item \textbf{If} $e' \in \Lcal_{\neq 2}$, then fix $e_2(i)=e'$, increase $i$ by $1$, set $j=-1$ and $\Unex(i+1,-1)=\Unex(i,j)\setminus \{e'\}$. Go back to Step (a).

 \textbf{Else}, set $e$ equal to the other half-edge $e''$  incident to the same vertex as $e'$. Increase $j$ by $1$ and set $\Unex(i,j+1)=\Unex(i,j)\setminus \{e',e''\}$. Go back to Step (b).
\end{enumerate}
\end{enumerate}

This algorithm produces the edge set of $K(\CMd)$ as the set of all the couples $\{e_1(i),e_2(i)\}$, since every time an half edge $e_1(i)$ is paired with a vertex $v$ of degree $2$, such vertex is removed, and $e_1(i)$ is paired with the other half-edge connected to $v$, exactly as in the construction of $K(\CMd)$. Moreover, every time the edge $e'$ is chosen uniformly over $\Unex(i,j)$, and  consequently, conditioning on $e' \in \Lcal_{\neq 2}$, its distribution is uniform over the set

\eqn{
\Unex(i,j) \cap  \Lcal_{\neq 2}= \Lcal_{\neq 2}\setminus \Big(\{e_1(i)\} \cup \bigcup_{h<i} \{e_1(h), e_2(h)\}\Big).
}
From this we obtain that the pairing $(\{e_1(i),e_2(i)\})_{i=1}^{\elltwo/2}$ is a uniform pairing over $\Lcal_{\neq 2}$, and thus $K(\CMd)\overset{d}{=}\CMdd$.\qed

\medskip
We next define the usual exploration process of the configuration model (see e.g. \cite[Section 1]{MolRee95}) that we are going to use multiple times in the paper.

\begin{defi}[Exploration of the Configuration Model]\label{defi:exploration} At each time $t$, we define the sets of half-edges $\{ \Acal_t , \Dcal_t , \Ncal_t \}$ (the active, dead and neutral sets), and run the following process:

\begin{itemize}
\item[{\tt Initalize}] Pick a vertex $v \in [n]$ according to any arbitrary (deterministic or stochastic) rule and we set all its half-edges as active. Set all the other half-edges as neutral.
\item[{\tt Step}] At each step $t$, pick a half-edge $e_1(t) $ in $\Acal_t$ uniformly at random, and pair it with another half-edge $e_2(t)$ chosen uniformly at random in $\Acal_t \cup \Ncal_t$. Set $e_1(t), e_2(t)$ as dead.\\
\textbf{If} $e_2(t) \in \Ncal_t$, then find the vertex $v(e_2(t))$ incident to $e_2(t)$ and activate all its other half-edges.
\item[{\tt Terminate}] If $\Acal_t=\varnothing$, terminate the process.
\end{itemize}
\end{defi}

 We use this algorithm to compute the number of vertices that are contained in cyclic components in the almost-two-regular graph:
\begin{prop}\label{prop:cycletot}Consider a sequence satisfying Condition \ref{2-up} or \ref{2-lo}. Define $\mathbf C(n)$ as the number of vertices of $\CMd$ in cyclic components. Then
\begin{equation}
\lim_{t \to \infty}\lim_{n \to \infty} \P(\mathbf C(n) \geq tn_2/\elltwo)=0.
\end{equation}
\end{prop}\proof
We run the exploration process described in Definition \ref{defi:exploration} starting from a uniformly chosen vertex of degree $2$. We run the process until the first of the two following stopping times:

\eqn{
T_{\neq 2}= \min\{t: e_2(t) \in \Lcal_{\neq 2}\}, \qquad T_C=\min\{t:e_2(t) \in \Acal_t\}.
}
It is impossible that $T_{\neq 2}=T_C$ because for every $t\leq T_{\neq 2}$, $\Acal_t \cap \Lcal_{\neq 2}=\varnothing$. Moreover, $v$ is in a cyclic component,  if and only if $T_C<T_{\neq 2}$.
We know that for every $t\leq \min\{T_{\neq 2},T_C\}$, $|\Acal_t|=2$ and $|\Lcal_{\neq 2} \cap \Ncal_t|=\elltwo$.  Consequently we have, for every $t\geq 1$

\eqn{
\P(T_{\neq 2}=t\mid  \min\{T_{\neq 2},T_C\}>t-1)=\elltwo\P(T_{C}=t\mid  \min\{T_{\neq 2},T_C\}>t-1),
}
so that
\eqn{
\P(T_{\neq 2}=t, T_C>T_{\neq 2})=\elltwo\P(T_C=t,T_{\neq 2}>T_C),
}
and thus
\eqan{
\nn\P(T_{\neq 2}>T_C)&=\sum_{t\geq 1}\P(\min\{T_{\neq 2},T_C\}=t) \P(T_{\neq 2}>T_C\mid \min\{T_{\neq 2},T_C\}=t)\\&=\sum_{t \geq 1}  \P(\min\{T_{\neq 2},T_C\}=t)\dfrac{1}{\elltwo+1}=\dfrac{1}{\elltwo+1}.
}
Consequently,
\begin{equation}
\E[\mathbf C(n)]=\sum_{v \in \Ncal_2}\P(v \text{ is in a cyclic component})=n_2 \P(T_{\neq 2}>T_C)=\dfrac{n_2}{\elltwo+1}.
\end{equation}
We thus write, by the first moment method,

\begin{equation}\begin{split}
\P(\mathbf C(n)>tn_2/\elltwo) \leq\dfrac{n_2}{\elltwo+1}\dfrac{\elltwo}{t n_2}=\dfrac{1}{t}+o(1),
\end{split}\end{equation}
and the claim follows.\qed

\medskip
We can now prove that under Condition \ref{2-up}, the largest component contains w.h.p. all the vertices of degree at least $3$, and is of size $n-\op(n)$. 

\begin{prop}Consider a sequence $\CMd$ that satisfies Condition \ref{2-up}. Then as $n \to \infty$,
\label{prop:ProofofMain1}
\begin{equation}
\lim_{t \to \infty}\lim_{n \to \infty} \P(n-|\Cmax| \geq tn_2/\elltwo)=0.
\end{equation}

Moreover, $[n] \setminus \Cmax$ contains only vertices of degree $2$ w.h.p..
\end{prop}
\proof
Every two vertices $v,w$ such that $d_v,d_w\neq 2$ are in the same component in $\CMd$ if and only if they are in the same component in $K(\CMd)$. Thus, if $K(\CMd)$ is connected, there exists a component $\tilde \cluster$ of $\CMd$ such that   $[n] \setminus \tilde \cluster$ contains only vertices of degree $2$. By Lemma \ref{lem:KernelDistribution}, $K(\CMd)$ is distributed as $\CMdd$, where $\dg'$ has minimum degree at least $3$. Consequently, $ \CMdd$ is connected w.h.p. by the main theorem of \cite{Wor81}.
 We thus write
\begin{equation}\begin{split}
\P(n-|\tilde \cluster|>tn_2/\elltwo) &\leq \P(\mathbf C(n)> tn_2/\elltwo)+\P(\Ncal_2 \cup \tilde \cluster \neq [n]) =\dfrac{1}{t}+o(1),
\end{split}\end{equation}
from this follows that w.h.p. $\tilde \cluster = \Cmax$ and so the claim follows.\qed

\medskip

This is enough to complete the proof of \eqref{main:first}. In order to prove \eqref{main:second}, we need to analyze the distribution of the number and size of the cyclic components.

From the previous analysis we see that w.h.p. probability all the components outside the giant are cycles, and that 
\begin{equation}
\sharp \text{vertices in cycles}= \Op\Big(\dfrac{n_2}{\elltwo}\Big).
\end{equation}

We know that if we remove the giant from a configuration model, the rest of the graph is distributed as a configuration model with the degree sequence of the remaining vertices. By Proposition \ref{prop:ProofofMain1} in this case what is left is w.h.p. a random $2$-regular graph.

 In a random $2$-regular graph, we know that $|\Cmax|/n$ converges in distribution to a random variable with no points with positive mass. Thus, we expect the second largest component of $\CMd$ (i.e. w.h.p. its largest cyclic component) to be of size $\Op \big(\frac{n_2}{\elltwo}\big)$ and to have a random size. To prove this we analyse the distribution of the number of cyclic components that are larger than $an_2/\elltwo$ for any fixed constant $a\in ( 0, \infty)$.

Define $C_n(k)$ as the number of cyclic components of size $k$ in $\CMd$. We will next analyse the process

\begin{equation}\label{ProcessDefinition}
S_n^{(a)}(k):=\left\{\begin{split} 0 \qquad \text{ if } k\leq an_2/\elltwo,\\
\sum_{j=an_2/\elltwo}^{k}C_n(j) \quad \text{if } k>an_2/\elltwo.
\end{split}\right.
\end{equation} 
In particular we are interested in proving the following proposition, from which we can deduce \eqref{main:second}:

\begin{prop}\label{prop:poiconv}Consider a sequence $\CMd$ satisfying $\elltwo/\ell \to 0$ and the associated processes $(S_n^{(a)}(k))_{k>0}$ as defined in \eqref{ProcessDefinition}. Then, for every $a<t<\infty$

\begin{equation}
S_n^{(a)}(tn_2/\elltwo)\dto \Poi\Big( \int_a^t \lambda (r)dr\Big),
\end{equation}
where $\lambda (t)=\dfrac{e^{-2t}}{2t}$.

\end{prop}
\proof
We prove the result through the method of moment. We define the factorial moments
\begin{equation}
(S_n^{(a)}(tn_2/\elltwo))_h=S_n^{(a)}(tn_2/\elltwo)\big(S_n^{(a)}(tn_2/\elltwo)-1\big)\cdots \big(S_n^{(a)}(tn_2/\elltwo)-h+1\big),
\end{equation}
to prove the Poisson convergence we need to prove that
\begin{equation}
\E[(S_n^{(a)}(tn_2/\elltwo))_h]\to \Big( \int_a^t \lambda (r)dr\Big)^h,
\end{equation}
for every $t>a$; $h \in \mathbb N$ (see e.g. \cite[Section 2.1]{Hofs17}).
We write
\begin{equation}
\E[(S_n^{(a)}(tn_2/\elltwo))_h]= \ddsum_{j_1,...,j_h\in \{an_2/\elltwo;tn_2/\elltwo\}}\E[C(j_1)C(j_2)\cdots C(j_h)].
\end{equation}
 We need to prove that

\begin{equation}\begin{split}\label{convexp}
\max_{j_1,...,j_h\in \{an_2/\elltwo;tn_2/\elltwo\}}&\E[C(j_1)C(j_2)\cdots C(j_h)]\Big(\dfrac{n_2}{\elltwo}\Big)^h\\&-\lambda(j_1n_2/\elltwo)\lambda(j_2n_2/\elltwo)\cdots \lambda(j_hn_2/\elltwo)\to 0
\end{split}\end{equation}
We define, for every $j \leq n_2$,
\eqn{
\Ccal (j):={\Ncal_2 \choose j},
}
 i.e., $\Ccal (j)$ is the set of all possible sets of vertices that can be arranged in a cycle of length $j$. We can thus write

\eqan{\label{eq:MultivariateMoment}
\E[C&(j_1)C(j_2)\cdots C(j_h)]\nonumber\\&=\sum_{A_1\in \Ccal(j_1)}\cdots\sum_{A_h\in \Ccal(j_h)}\E[\indi_{C_{A_1}}\cdots\indi_{C_{A_h}}]\\&=\sum_{A_1\in \Ccal(j_1)}\cdots\sum_{A_{h-1}\in \Ccal(j_{h-1})}\E[\indi_{C_{A_1}}\cdots\indi_{C_{A_{h-1}}} ]\sum_{A_{h}\in \Ccal(j_{h})}\E[\indi_{C_{A_h}}\mid C_{A_1},\dots,C_{A_{h-1}}].\nonumber
}
We note immediately that, if there exist $i, l \leq h$ such that $A_i\cap A_l\neq \varnothing$, then $\indi_{C_{A_i}}\indi_{C_{A_l}} =0$ deterministically. Consequently, we can restrict the sum to the case in which the sets $A_1,...,A_h$ are all mutually disjoint.
Define $m_i:=\sum_{l=1}^{i-1}j_l$, that is, as the number of vertices used to create the first $i-1$ cycles, all of which have degree $2$.
We note that for all $1\leq i \leq h$, as long as $A_{i}$ does not overlap with any of the previous sets, $\E[\indi_{C_{A_i}}\mid C_{A_1},\dots,C_{A_{i-1}}]$ is the probability that the vertices in $A_i$ form a cyclic component in a configuration model $\CMdt$, with $\tilde\dg$ satisfying $\tilde n_2=n_2-m_i$ and $\tilde n_l=n_l$ for all $l\neq 2$, where $\tilde d_l$ is the number of vertices of degree $l$ in $\tilde \dg$. We further note that $\tilde\dg$ does not depend on the precise choice of $A_1,...,A_{i-1}$, as long as $A_l \in \Ccal (j_l)$ for every $l$.
 Thus, for all possible disjoint collections $A_1,...,A_h$, for every $i\leq h$,
\begin{equation}\begin{split}
\E[&\indi_{C_{A_i}}\mid C_{A_1},\dots,C_{A_{i-1}}]\\&=\dfrac{2j_i-2}{\ell-2m_i-1}\dfrac{2j_i-4}{\ell-2m_i-3}\cdots\dfrac{2}{\ell-2m_i-2k_i+3}\dfrac{1}{\ell-2m_i-2k_i+1}
\\&=2^{j_i-1}(j_i-1)!\prod_{l=0}^{j_i-1}\dfrac{1}{\ell-2m_i-2l-1}
\end{split}\end{equation}
There exists ${n_2-m_i \choose j_i}$ distinct choices for $A_i$ that do not overlap with any $A_l$ for any $l<i$, so we obtain
\begin{equation}\label{eq:prodex}\begin{split}
\sum_{A_i\in \Ccal(j_i)} \E[\indi_{C_{A_i}}\mid C_{A_1},\dots,C_{A_{i-1}}]&={n_2-m_i \choose j_i}2^{j_i-1}(j_i-1)!\prod_{g=0}^{j_i-1}\dfrac{1}{\ell-2m_i-2g-1}\\
&=\dfrac{1}{2j_i}\prod_{g=0}^{j_i-1}\dfrac{2(n_2-m_i-g)}{\ell-2m_i-2g-1}
\end{split}\end{equation}

We use that, for every $g\geq 0$,

\begin{equation}
\dfrac{2(n_2-m_i-g)}{\ell-2m_i-2g-1}\leq \dfrac{2n_2}{\ell-1}.
\end{equation}
For a lower bound we write, recalling that $m_i+g\leq h t\ell/\elltwo$,
\begin{equation}\begin{split}
 \dfrac{2n_2}{\ell}-\dfrac{2(n_2-m_i-g)}{\ell-2m_i-2g-1}&\leq  \dfrac{2n_2}{\ell}-\dfrac{2n_2-2(m_i+g+1))}{\ell-2(m_i+g+1)}\\&=\dfrac{\ell(2n_2-2(m_i+g+1))-2n_2(\ell-2(m_i+g+1))}{\ell(\ell-2(m_i+g+1))}\\
 &=\dfrac{2(m_i+g+1)(\ell-2n_2)}{\ell^2(1-o(1))}\leq \frac{ht}{\ell}.
\end{split}\end{equation}
Thus we obtain that 
\begin{equation}\begin{split}
\prod_{g=0}^{j_i-1}\dfrac{2(n_2-m_i-h)}{\ell-2m_i-2g-1}&\geq \Big(\dfrac{2n_2-ht}{\ell}\Big)^{j_i-1}=\Big(1-\dfrac{\elltwo+ht}{\ell}\Big)^{j_i-1}\\&=\Big(1-\dfrac{\elltwo+ht}{\ell}\Big)^{\frac{\ell}{\elltwo+ht}\frac{(j_i-1)(\elltwo+ht)}{\ell}},
\end{split}\end{equation}
and
\begin{equation}\begin{split}
\prod_{g=0}^{j_i-1}\dfrac{2(n_2-m_i-h)}{\ell-2m_i-2g-1}&\leq \Big(\dfrac{2n_2}{\ell-1}\Big)^{j_i-1}=\Big(1-\dfrac{\elltwo}{\ell-1}\Big)^{j_i-1}\\&=\Big(1-\dfrac{\elltwo}{\ell-1}\Big)^{\frac{\ell-1}{\elltwo}\frac{(j_i-1)\elltwo}{\ell-1}}.
\end{split}\end{equation}
We thus write
\eqan{
- \frac{(j_i-1)\elltwo}{\ell-1}(1+o(1))&\leq \log \big( 2j_i\E[\indi_{C_{A_i}}\mid C_{A_1},\dots,C_{A_{i-1}}]\big) \\&\leq-\frac{(j_i-1)(\elltwo+ht)}{\ell}(1+o(1)),
}
where the $o(1)$ term is independent of $j_i$. We can thus conclude that
\eqan{
\max_{j_i \in \{an_2/\elltwo;tn_2/\elltwo\}}\max_{A_1,...,A_{i-1}}\Big| 2j_i\E[\indi_{C_{A_i}}\mid C_{A_1},\dots,C_{A_{i-1}}]-\e^{-j_i\elltwo/\ell}\Big|\to 0
}
Recall that 
\eqn{
\dfrac{\e^{-j_i\elltwo/(2n_2)}}{2j_i}=\lambda(j_i\elltwo/n_2)\frac{\elltwo}{n_2}
.}
We thus obtain, 
\eqn{\label{eq:unicon}
\max_{j_i \in \{an_2/\elltwo;tn_2/\elltwo\}}\max_{A_1,...,A_{i-1}}\Big|\E[\indi_{C_{A_i}}\mid C_{A_1},\dots,C_{A_{i-1}}]\frac{n_2}{\elltwo}-\lambda(j_i\elltwo/n_2)\Big| \to 0.
}
Iterating the substitution of \eqref{eq:unicon} in \eqref{eq:MultivariateMoment} for all the other $j_i, i<h$, we obtain \eqref{convexp}.
By the uniform convergence we obtain
\eqan{
\ddsum_{j_1,...,j_h\in \{an_2/\elltwo;tn_2/\elltwo\}}&\E[C(j_1)C(j_2)\cdots C(j_h)]\\
=&\Big(\frac{\elltwo}{n_2}\Big)^h(1+o(1))\ddsum_{j_1,...,j_h\in \{an_2/\elltwo;tn_2/\elltwo\}} \lambda(j_1\elltwo/n_2)\cdots\lambda(j_i\elltwo/n_2)\nonumber\\
\nonumber \\ \to& \Big( \int_a^t \lambda (r)dr\Big)^h.\nonumber
}
This, by the method of moments, yields the claim.
\qed
\medskip

We now put together the results obtained in this section to prove Theorem \ref{thm:MainUp}

\begin{proof}[Proof of Theorem \ref{thm:MainUp}]

Proposition \ref{prop:ProofofMain1} directly implies \eqref{main:first}, so what we are left to prove is \eqref{main:second}.

Define $\Cmaxc$ as the largest cyclic component. By Proposition \ref{prop:ProofofMain1}

\eqn{\label{eq:SecondCyclic}
\lim_{n \to \infty}\P(\Cmaxc=\cluster_{2})=1.
}
By Proposition \ref{prop:poiconv}, for any $a\in (0,\infty)$, 
\eqan{\label{eq:PoiZero}
\nonumber\lim_{n \to \infty}\P(|\Cmaxc|\leq an_2/\elltwo)&\leq \lim_{t \to \infty}\lim_{n \to \infty} \P(S_n^{(a)}(tn_2/\elltwo)=0)\\&=\lim_{t \to \infty}  \P\Big(\Poi\Big(\int_{a}^t\lambda(r)dr\Big)=0\Big)\\&=\exp\Big\{-\int_a^\infty \frac{\e^{-2r}}{2r}dr\Big\}.\nonumber
}
For the matching lower bound we write, using Proposition \ref{prop:cycletot},

\eqan{\label{eq:PoiZero}
\lim_{n \to \infty}\P(|\Cmaxc|\leq an_2/\elltwo)&\geq \lim_{t \to \infty}\lim_{n \to \infty} \P(S_n^{(a)}(tn_2/\elltwo)=0) \nonumber\\&\qquad-\lim_{t \to \infty}\lim_{n \to \infty} \P(\mathbf C(n)\geq t(n_2/\elltwo))\\&=\lim_{t \to \infty} \P\Big(\Poi\Big(\int_{a}^t\lambda(r)dr\Big)=0\Big)+o(1)\nonumber\\&=\exp\Big\{-\int_a^\infty \frac{\e^{-2r}}{2r }dr\Big\}.\nonumber
}
Combining \eqref{eq:SecondCyclic} and \eqref{eq:PoiZero} and the fact that $n_2=n(1-o(1))$, the claim follows.
\end{proof}

\section{Proof of Theorem \ref{thm:MainLow}}

In this section we analyze the lower almost-2-regular graph, and prove Theorem \ref{thm:MainLow}. Since Conditions \ref{2-lo} allow only for vertices of degree $1$ and $2$, we know that all the components in $\CMd$ are either lines (components made of $2$ vertices of degree $1$ connected by vertices of degree $2$) or cycles, if we consider vertices of degree $2$ with a self loop cycles of length $1$ and double edges between vertices of degree $2$ cycles of length $2$.
Thus

\eqn{
|\Cmax|= \max\{ |\Cmaxc|,|\Cmaxl|\},
}
where $\Cmaxc$ is the largest cyclic component and $\Cmaxl$ is the largest line component. Moreover, define $\Cjl$ as the $j$-th largest line component.
%
%
%

We next prove that the size of the largest line components concentrates and it is w.h.p. larger than that of $\Cmaxc$, which we can estimate using Proposition \ref{prop:cycletot}.

\begin{lem}\label{lem:Cmaxl}Consider $\CMd$ satisfying Conditions \ref{2-lo}. Then, for every $j \in \mathbb N$,

\eqn{
\dfrac{n_1|\Cjl| }{n \log n_1} \pto 2.
}
\end{lem}
\proof
We start by proving the upper bound. For every $\alpha \in (0, \infty)$ we define
\eqn{N(1,\alpha):= \sharp \{v \in \Ncal_1: |\cluster (v)|> \alpha n\log n_1/n_1\},
}
where for every $v \in [n]$, $\cluster (v)$ is the connected component that contains $v$. 
To find bounds on $N(1,\alpha)$ we run the exploration from Definition \ref{defi:exploration} starting from a uniformly chosen vertex $v \in \Ncal_1$. Since there are no vertices of degree larger than $2$, at every step $t$, $|\Acal_t|\in \{0,1\}$, and the exploration of $\cluster (v)$ ends as soon as $e_2(t) \in \Lcal_{\neq 2}$, i.e., the first time the exploration finds another vertex of degree $1$. Thus, the only way the process can survive up to  time $\alpha n\log n_1/n_1$ is that $e_2(t) \notin \Lcal_{\neq 2}$, for all $t \leq \alpha n\log n_1/n_1$. We write the probability for this to happen as

\begin{equation}\begin{split}
\pi_\alpha:=\P(|\cluster (v)|> \alpha n\log n_1/n_1)&=\prod_{t=0}^{\alpha n \log n_1/n_1} \Big(\dfrac{2(n_2-t)}{\ell-2t-1}\Big)\\&=\prod_{k=0}^{\alpha n \log n_1/n_1}\Big(1-\dfrac{n_1-1}{\ell-2t-1}\Big).
\end{split}\end{equation}
We bound this quantity by

\begin{equation}\begin{split}
\pi_\alpha\leq \Big(1-\dfrac{n_1-1}{\ell}\Big)^{\alpha n \log n_1/n_1}=\e^{-\log n_1/2(1+o(1))}
\end{split}\end{equation}
\begin{equation}\begin{split}
\pi_\alpha\geq \Big(1-\dfrac{n_1-1}{\ell- 2\alpha n\log n_1/n_1}\Big)^{\alpha n \log n_1/n_1}=\e^{-\log n_1/2(1+o(1))}
\end{split}\end{equation}
Consequently,
\begin{equation}\label{eq:firstalpha}\begin{split}
\E[N(1,\alpha)]&= \sum_{v \in \Ncal_1}\P(|\cluster (v)|> \alpha n\log n_1/n_1)\\&=n_1\pi_\alpha=n_1\e^{-\alpha\log n_1/2(1+o(1))}=n_1^{(2-\alpha)/2(1+o(1))} .
\end{split}\end{equation}
By the first moment method,
\eqn{
\P\Big(\dfrac{n_1|\Cmaxl| }{n \log n_1} \geq \alpha\Big)\leq \E[N(1,\alpha)]/2,
}
so that, for every $\varepsilon >0$,
\eqn{
\lim_{n \to \infty}\P\Big(\dfrac{n_1|\Cmaxl| }{n \log n_1} \geq 2+\varepsilon\Big)=0.
}

Next we prove a sharp lower bound on the size of $|\Cjl|$ for every $j \geq 1$. Using the Chebychev inequality we write

\eqn{\label{eq:Chebychev}
\P\Big(\Cjl \leq \frac{\alpha n \log n_1}{n_1}\Big)=\P(N(1,\alpha)\leq 2j) \leq \dfrac{\E[N(1,\alpha)^2]-\E[N(1,\alpha)]^2}{(\E[N(1,\alpha)]-2j)^2}.
}
By \eqref{eq:firstalpha}, we know that for every $\alpha <2$, $\E[N(1,\alpha)]\to \infty$.
Thus, it is enough to prove that for every $\alpha<2$, 
\eqn{
\dfrac{\E[N(1,\alpha)^2]}{\E[N(1,\alpha)]^2}\to 1.
}
We write
\eqn{\label{eq:secondmoment}\begin{split}
\E[N(1,\alpha)^2]&=\sum_{v,w \in \Ncal_1}\P(|\cluster (v)|,|\cluster (w)|> \alpha n\log n_1/n_1)
\\&=\sum_{v \in \Ncal_1}\P(|\cluster (v)|> \alpha n\log n_1/n_1)\\
&\qquad +\sum_{v,w \in \Ncal_1, v \neq w}\P(|\cluster (v)|,|\cluster(w)|\geq \alpha n\log n_1/n_1, \cluster(v)\neq\cluster(w))\\
& \qquad+\P(|\cluster (v)|\geq \alpha n\log n_1/n_1, \cluster(v)=\cluster(w)).
\end{split}}

We know that almost surely for every $v \in \Ncal_1$ there exists exactly $1$ another vertex $w \in \Ncal_1 \cap \cluster (v)$, so that, for every $v \in \Ncal_1$,

\eqn{\begin{split}\label{eq:samecomp}
\sum_{w \in \Ncal_1, v \neq w}&\P(|\cluster (v)|\geq \alpha n\log n_1/n_1, \cluster(v)=\cluster(w))\\&=\sum_{v \in \Ncal_1}\P(|\cluster (v)|> \alpha n\log n_1/n_1)= \pi_\alpha.
\end{split}}

To bound the probability that $v$ and $w$ are in distinct large components, we now run two copies of the exploration process from Definition \ref{defi:exploration}, starting from two different vertices $v$ and $w$ in $\Ncal_1$. We first explore starting from $v$, and we let the exploration run up to time $\alpha n\log n_1/n_1$. If the exploration has survived, then we know that $|\cluster (v)|> \alpha n\log n_1/n_1$, and that $\Acal_{ \alpha n\log n_1/n_1}=\{e'\}$ for some half-edge $e' \notin \Lcal_{\neq 2}$. In this case, we start start running a new exploration from $w$. We know that $|\cluster (w)|> \alpha n\log n_1/n_1$, $\cluster (w) \neq \cluster (v)$ if the exploration starting from $w$ survives up to time $\alpha n\log n_1/n_1$ without finding  neither $e'$ nor any half-edge in $ \Lcal_{\neq 2}$.
%

We bound this probability by
\eqn{
\prod_{k=0}^{\alpha n \log n_1/n_1}\Big(1-\dfrac{n_1-2}{\ell-2k-2\alpha n\log n_1/n_1}\Big)\leq \prod_{k=0}^{\alpha n \log n_1/n_1}\Big(1-\dfrac{n_1-1}{\ell-2k-1}\Big)=\pi_\alpha ,
}
so that 

\eqn{\label{eq:diffcomp}
\P(|\cluster (v)|,|\cluster(w)|\geq \alpha n\log n_1/n_1, \cluster(v)\neq\cluster(w))\leq \pi_\alpha^2.
}
We thus obtain, substituting \eqref{eq:firstalpha}, \eqref{eq:samecomp} and \eqref{eq:diffcomp} into \eqref{eq:secondmoment}

\eqan{
\E[N(1,\alpha)^2]&\leq 2 \pi_\alpha n_1 + \pi_\alpha^2 n_1(n_1-1)
}
We further compute, recalling that $n_1 \pi_\alpha= \E[N(1,\alpha)]$,
\eqan{\label{eq:secondalpha}
\E[N(1,\alpha)^2]&\leq 2\E[N(1,\alpha)]+\E[N(1,\alpha)]^2
}
If $\alpha=2-\varepsilon$, then, by \eqref{eq:firstalpha}, $\E[N(1,\alpha)]\to \infty$, so we obtain
\eqn{
\E[N(1,\alpha)^2]=\E[N(1,\alpha)]^2(1+o(1)).
}
Consequently, by \eqref{eq:Chebychev}, for every $j \geq 1$,
\eqn{
\lim_{n\to \infty}\P\Big(\dfrac{n_1|\Cjl| }{n \log n_1} \leq 2-\varepsilon\Big)=\lim_{n\to \infty}\P(N(1,2-\varepsilon)\leq 2j) =0.\qed
}

\medskip
We can now finally prove Theorem \ref{thm:MainLow}.

\begin{proof}[Proof of Theorem \ref{thm:MainLow}]

By Lemma \ref{lem:Cmaxl}, for every $\alpha <2$, $j \in \mathbb N$,

\eqan{
\lim_{n \to \infty} \P\Big(|\cluster_j|\geq\dfrac{\alpha n \log n_1}{n_1 }\Big)\geq \lim_{n \to \infty} \P\Big(|\Cjl|\geq\dfrac{\alpha n \log n_1}{n_1 }\Big)=1.
}
On the other hand, by Proposition \ref{prop:cycletot} and Lemma \ref{lem:Cmaxl}, for every $\alpha >2$, $j \in \mathbb N$,
\eqan{
\lim_{n \to \infty} &\P\Big(|\cluster_j|\geq \dfrac{\alpha n \log n_1}{n_1 }\Big)\\&\leq \lim_{n \to \infty} \Big(\P\Big(|\Cmaxl|\geq \dfrac{\alpha n \log n_1}{n_1 }\Big)+\P\Big(|\Cmaxc|\geq \dfrac{\alpha n \log n_1}{n_1 }\Big)\Big) =0,\nonumber
}
bringing together upper an lower bound, we obtain the claim.
\end{proof}
	\section*{Acknowledgments}

The work in this paper is supported by the Netherlands Organisation for Scientific Research 
(NWO) through Gravitation-grant NETWORKS-024.002.003 and by the European Research Council (ERC) through Starting Grant Random Graph, Geometry and Convergence 639046.

\begin{small}
\bibliographystyle{abbrv}
\bibliography{LorenzosBib}

\end{small}

%
%
%
%
%

\end{document}